\documentclass[12pt,a4paper]{amsart}
\usepackage{graphicx,amssymb}
\usepackage{amsmath,amssymb,amscd}
\input xy
\xyoption{all}
\usepackage[all]{xy}
\usepackage{hyperref}

\newcommand{\ra}{\rightarrow}

\newcommand{\ZZ}{\mathbb Z}

\newcommand{\PP}{\mathbb P}

\newcommand{\cO}{\mathcal{O}}

\newcommand{\cC}{\mathcal{C}}
\newcommand{\cL}{\mathcal{L}}

\newcommand{\Pic}{\mbox{Pic}}

\theoremstyle{plain}
\newtheorem{theorem}{Theorem}[section]
\newtheorem{lem}[theorem]{Lemma}
\newtheorem{prop}[theorem]{Proposition}
\newtheorem{cor}[theorem]{Corollary}
\newtheorem{defin}[theorem]{Definition}
\newtheorem{rem}[theorem]{Remark}
\newtheorem{ex}[theorem]{Example}
\newtheorem{qu}[theorem]{Question}
\numberwithin{equation}{section}
\begin{document}
\title[Gonality sequence]{On the gonality sequence of an algebraic curve}

\author{H. Lange}
\author{G. Martens}

\address{H. Lange\\Department Mathematik\\
              Universit\"at Erlangen-N\"urnberg\\
              Bismarckstra\ss e $1\frac{ 1}{2}$\\
              D-$91054$ Erlangen\\
              Germany}
              \email{lange@mi.uni-erlangen.de}
\address{G. Martens\\Department Mathematik\\
              Universit\"at Erlangen-N\"urnberg\\
              Bismarckstra\ss e $1\frac{ 1}{2}$\\
              D-$91054$ Erlangen\\
              Germany}
              \email{martens@mi.uni-erlangen.de}

\thanks{}

\begin{abstract}
For any smooth irreducible projective curve $X$, the gonality sequence $\{ d_r \;| \; r \in \mathbb N \}$ is a strictly increasing sequence 
of positive integer invariants of $X$. In most known cases $d_{r+1}$ is not much bigger than $d_r$. In our terminology this means 
the numbers $d_r$ satisfy the slope inequality. It is the aim of this paper to study cases when this is not true. We
give examples for this of extremal curves in $\PP^r$, for curves on a general K3-surface in $\PP^r$ and for complete 
intersections in $\PP^3$.  
\end{abstract}

\maketitle

\section{Introduction}

Let $X$ be a complex smooth irreducible projective curve of genus $g \geq 4$. If we understand by the degree of a rational 
map $X \ra \PP^r$
the product of the usual degree of the map onto its image times the degree of the closure of its image in $\PP^r$, 
then, for any positive integer $r$, the invariant $d_r = d_r(X)$ 
of $X$ is defined to be the smallest number $d$ such that $X$ admits a nondegenerate rational map of degree $d$ into $\PP^r$. 
The invariant $d_1$ is the usual gonality of $X$. Therefore the sequence
$d_1,d_2,d_3,\ldots$ is called the {\it gonality sequence} of $X$.

For any curve and any $r \geq g$, the numbers $d_r$ are known by Riemann-Roch. Hence there are only 
finitely many interesting numbers in a gonality sequence.
By Brill-Noether theory the whole sequence is known for general curves of genus $g$.                                    
Moreover it is known for smooth plane, hyperelliptic, trigonal, general tetragonal, general pentagonal 
and bielliptic curves 
(for the references see \cite{ln} and Proposition \ref{prop4.2} below). 
Apart from that the knowledge on the numbers $d_r$ seems to be scarce. It is the aim of this note to
give a first systematic investigation of these numbers.

Clearly the gonality sequence plays an essential r${\hat o}$le in the theory of special curves. Recently it turned out that
it also is important for Brill-Noether theory of vector bundles on $X$. In fact, most of the proofs in \cite{ln} use how some of 
the numbers $d_r$
are related to each other. In ``most'' cases they satisy the inequality
$$
\frac{d_r}{r} \geq \frac{d_{r+1}}{r+1}.
$$
It is true for all $r$ and all the examples of curves mentioned above, apart from a few numbers $r$ for smooth plane curves
(see Proposition \ref{prop3.2}). 
We call it the {\it slope inequality} (for the gonality sequence) and it was the main motivation for us to find more counter-examples to 
its validity for all $r$. \\

In Corollary \ref{cor3.4} we show that the slope inequality is violated for any $r \geq 2$ and any extremal 
curve of degree $d = 3r-1$ in $\PP^r$. More generally, we prove that the curves of the following three classes
do not satisfy all slope inequalities:
\begin{enumerate}
\item extremal curves of degree $\geq 3r-1$ in $\PP^r$ ($r \geq 2$; Theorem \ref{thm3.10}),
\item smooth curves of degree $\geq 8(r-1)$ on a general K3-surface of degree $2(r-1)$ in $\PP^r$ ($r\geq 3$; Theorem \ref{thm5.2}),
\item smooth complete intersections of surfaces of degree $s$ and $p$ with $2 \leq s \leq p$ and $p \geq 4$ 
in $\PP^3$ (Theorem \ref{thm6.1}).
\end{enumerate}
 
Section 2 contains some preliminaries, mainly on extremal curves. In Section 3 we recall some results on the gonality sequence and 
prove some new ones. Section 4 contains the proofs of the above mentioned results concerning extremal curves, 
as well as some consequences and examples. Finally in Section 5 respectively 6 we 
prove the results on curves on a general K3-surface respectively on complete intersection curves in $\PP^3$. Moreover, 
the case of complete intersection curves in $\PP^3$ is generalized to Halphen curves in $\PP^3$.

\section{Preliminaries}

Let $X$ denote an irreducible smooth projective curve of genus $g \geq 1$ over the field of complex numbers.
We recall Castelnuovo's genus bound, following \cite{a2}, \cite{acgh} or \cite{ci}: 
Let $g_d^r$ be a simple linear series on $X$ such that $d > r > 1$, and let $\chi(d,r)$ denote the unique positive integer
satisfying
$$
\frac{d-1}{r-1} - 1 \leq \chi(d,r) < \frac{d-1}{r-1}.
$$
Then we have for the genus $g$ of $X$,
\begin{equation}  \label{eqcast}
g \leq \pi(d,r) :=  \chi(d,r) \left(\frac{\chi(d,r) - 1}{2}(r-1) + \epsilon \right) 
\end{equation}
where we write $d-1 = \chi(d,r)(r-1) + \epsilon$ with $1 \leq \epsilon \leq r-1$. If the upper bound is attained,
the $g_d^r$ is complete and very ample and the curve $X \subset \PP^r$ is called an {\it extremal curve}. 

If $d < 2r$, the $g_d^r$ is non-special and we obtain $\pi(d,r) = d-r$. If $2r \leq d < 3r-1$, one computes 
$\pi(d,r) =2d -3r + 1$. In particular we have $\pi(d,r) < d$ for $d < 3r-1$. 

In view of Theorem \ref{thm3.10} below we are interested in the set of genera of extremal curves of degree
$d \geq 3r-1$ in $\PP^r$. So let $d \geq 3r-1$. By definition of $\chi := \chi(d,r)$ this is equivalent to $\chi \geq 3$,

According to \eqref{eqcast} $\chi$ divides $\pi(d,r)$ if $\chi$ is odd, and $\frac{\chi}{2} \geq 2$ divides $\pi(d,r)$ if 
$\chi$ is even. For $r \geq 3$ we have
$$
\frac{\chi - 1}{2}(r-1) + \epsilon \geq \frac{\chi - 1}{2} \cdot 2 + 1 = \chi \geq 3,
$$
and for $r=2$ we have 
$$
3 \leq \chi = d-2 \quad \mbox{and} \quad \frac{\chi - 1}{2}(r-1) + \epsilon = \frac{d-1}{2} \geq 2.
$$
In particular, $\pi(d,r)$ cannot be a prime number, and $\pi(d,r) \neq 1,4,8$. We claim that $\pi(d,r) \neq 14$, too.
This is clear for $r=2$ and for $r \geq 3$ we have $\pi(d,r) \geq \chi^2$. So $\chi$ would have to be 7 for odd $k$ and 
$\chi \geq 4$ for even $k$.\\ 

Now consider the set
$$
C := \{ \pi(d,r) \;|\; r \geq 2 \; \mbox{and} \; \chi(d,r) \geq 3 \}
$$
of genera of extremal curves of degree $d \geq 3r-1$ in $\PP^r$. 

\begin{prop} \label{genera}
Let $3 \leq g \in \ZZ,\; g \neq 4,8,14$. Then $g \not \in C$ implies that $g$ is odd with at most two different prime factors.
\end{prop}

\begin{proof}
First we show that any even positive integer $g \neq 2,4,8,14$ is in $C$.

In fact, for $d = 3r-1, \; r \geq 2$ we have $\pi(d,r) = 3r = 6,9,12, \cdots$, for $d = 4r-2, \; r \geq 2$ we obtain
$\pi(d,r) = 6r -2 = 10, 16,22, \cdots$ and for $d = 4r-1, \; r \geq 3$ we get $\pi(d,r) = 6r+2 = 20,26,32,\cdots$.

Next we prove the following assertion: Suppose $g \geq 3$ is an odd integer which is not a prime number, and 
let $p$ be the smallest prime divisor of $g$. If $g \geq \frac{1}{4}p^2(p+1)$, then $g \in C$.

Note that the assertion implies that any positive odd integer $g$ with at least $3$ (not necessarily distinct) prime divisors
lies in $C$.

To prove the assertion, let $g = p \cdot \alpha$ with $p \leq \alpha \in \ZZ$.
For an integer $r \geq 2$ we can write 
$$
\frac{p-1}{2}(r-1)+1 \leq \alpha \leq \frac{p-1}{2}(r-1) + (r-1) = \frac{p+1}{2}(r-1)
$$
if and only if
$$
\frac{2 \alpha}{p+1} \leq r-1 \leq \frac{2\alpha -2}{p-1}.
$$
We can find such an integer $r$ if and only if 
$$
\lceil \frac{2 \alpha}{p+1} \rceil \leq \lfloor \frac{2 \alpha -2}{p-1} \rfloor.
$$ 

We have $\lfloor \frac{2 \alpha -2}{p-1} \rfloor = \frac{2 \alpha -2 -i}{p-1} \geq \frac{2 \alpha -2 - (p-2)}{p-1}$ 
for some $i, \; 0 \leq i \leq p-2$. Thus, if
\begin{equation} \label{equat2.2}
\frac{2 \alpha -p}{p-1} \geq \frac{2 \alpha}{p+1},
\end{equation}
we have $\lfloor \frac{2 \alpha -2}{p-1} \rfloor \geq \frac{2 \alpha}{p+1}$ and thus
$\lfloor \frac{2 \alpha -2}{p-1} \rfloor \geq \lceil \frac{2 \alpha}{p+1} \rceil$. So \eqref{equat2.2} implies that we 
can find an integer $r$ we want, and if we set $\epsilon := \alpha - \frac{p-1}{2}(r-1)$, we have 
$$
g = p \alpha = p \left(\frac{p-1}{2}(r-1) + \epsilon \right) \quad  \mbox{with}  \quad 1 \leq \epsilon \leq r-1.
$$
Setting $d := p(r-1) + 1 + \epsilon$, we can find an extremal curve of degree $d$ in $\PP^r$ of genus $g$ (with $\chi = p \geq 3$),
i.e. $g \in C$.

Obviously \eqref{equat2.2} is equivalent to $4 \alpha \geq p(p+1)$. So this proves the assertion.

We have already seen that the assertion implies that an odd positive integer $g \not \in C$ can have at most two prime factors.
Now if $g = p^2$ ($p$ an odd prime), we may take an extremal space curve of degree $d = 2p+2$. This curve has genus $p^2$, 
showing that $p^2 \in C$.
\end{proof}

\begin{rem}
{\em 
The proof of the assertion shows that if $p, p+2$ are twin primes, then $p(p+2) \not \in C$, unless $p \leq 5$.
If, however, $p$ and $q$ are odd primes such that $q \geq \frac{p(p+1)}{4}$, then $pq \in C$ by the assertion.
$  \hspace{1.8cm} \square$  }
\end{rem}

The paper \cite{ci} is somewhat difficult to obtain. So for the convenience of the reader 
we close this section by recalling two results of it on curves in $\PP^r$ of ``sufficiently high'' genus, 
which will be needed in the sequel.
For the following result see \cite[Teorema 2.11]{ci}.

\begin{theorem} \label{thm2.3}
Let $g_d^r$ and $g_m^s$ be linear series on $X$  such that $g_d^r$ is base point free and simple,
$$
m \leq kd \quad \mbox{and} \quad  s \geq {k+1 \choose 2}(r-1) + k
$$
for some integer $k,\; 1 \leq k < \chi(d,r)$. Then 
$$
g_m^s = kg_d^r \quad \mbox{or} \quad g \leq \pi(d,r) - \chi(d,r) +k.
$$
\end{theorem}

The second result is part of \cite[Osservatione 2.19]{ci}.

\begin{prop} \label{prop2.4}
Let $g_d^r$ be a base point free and simple linear series on $X$ such that
$$
g \geq \pi(d,r) - \chi(d,r) + 2.
$$
Then the $g_d^r$ is complete, $\dim(2g_d^r) = 3r-1$ and $X$ is mapped into a surface of degree $r-1$ in $\PP^r$
via the morphism given by the $g_d^r$. 
\end{prop}

\section{The gonality sequence}

Let now $g \geq 4$. For 
any positive integer $r$ the invariant $d_r$ of $X$ is defined as
$$
d_r = d_r(X) := \min \{ d \; |\; X \; \mbox{admits a linear series} \; g_d^r \}.
$$  
So $d_1$ is the gonality of $X$, $d_2$ is the minimal degree of a non-degenerate rational map $X \ra \PP^2$ etc.
The sequence $d_1,d_2,d_3,\cdots$ is called the {\it gonality sequence} of the curve $X$. The {\it Clifford index} 
of $X$ is defined as usual as
$$
\gamma := \min \{ d - 2r \; | \; X \; \mbox{admits a} \; g_d^r \; \mbox{with} \; r \geq 1 \; \mbox{and} \; d \leq g-1     \}.
$$
We say that a linear series $g_d^r$ {\it contributes to} $\gamma$ if $r \geq 1$ and $d \leq g-1$. We say that $d_r$
(respectively a $g_d^r$ or a line bundle defining it) {\it computes} $\gamma$ if in addition $\gamma = d_r - 2r$ (respectively 
$\gamma = d -2r$). 
For the following lemma see \cite[Lemmas 4.2 and 4.3]{ln}.

\begin{lem} \label{lem2.1}
{\em (a)} $d_r < d_{r+1}$ for all $r$;\\
{\em (b)} if a line bundle $L$ computes $d_r$, then $h^0(L) = r+1$ and $L$ is generated;\\  
{\em (c)} $d_{r+s} \leq d_r + d_s$ for any $r,s \geq 1$; in particular 
$$
d_r \leq r \cdot d_1 \leq r \cdot \frac{g+3}{2} \quad \mbox{for any} \quad  r;
$$
{\em (d)} If $d_r + d_s = d_{r+s}$, then $d_n = nd_1$ for all $n \leq r+s$.
\end{lem}

\begin{lem} \label{lem2.2}
{\em (a)} $d_r = r+g$ for $r \geq g$;\\
{\em (b)} $d_r = r + g -1$ for $g > r > g - d_1$;\\
{\em (c)} $d_r \leq g - \left[ \frac{g}{r+1} \right] + r$ and for a general curve we have
$$
d_r = g - \left[ \frac{g}{r+1} \right] + r;
$$
{\em (d)} $d_r \geq \min \{ \gamma +2r, g+r -1 \}$ for all $r$.
\end{lem}

\begin{proof}
Apart from (b) this follows from the Riemann-Roch theorem, the definition of the Clifford index and Brill-Noether theory.

Proof of (b): Assume that $\dim |K_X - g_{d_r}^r | \geq 1$. Then, by definition of $d_1$, 
$\deg |K_X - g_{d_r}^r | = 2g -2 -d_r \geq d_1$. Since $g_{d_r}^r$ contributes to $\gamma$, and from \cite[Theorem 2.3]{cm}
we know that $\gamma \geq d_1 - 3$. Consequently,
$$
2g-2 - d_1 \geq d_r \geq 2r + \gamma \geq 2r + d_1 - 3,
$$
which implies $r \leq g - d_1$. Hence, for $g > r > g-d_1$ we obtain $\dim |K_X -g_{d_r}^r | \leq 0$. 
But $g_{d_r}^r$ is certainly a special linear system. So $\dim |K_X -g_{d_r}^r | = 0$. 
Now Riemann-Roch
implies $d_r = r + g-1$.  
\end{proof}

For specific curves, even of ``small'' genus, it is in general not easy to compute its gonality sequence.

\begin{ex} \label{ex3.3}
{\em Let $X$ be a curve of genus $14$ with a $g_{13}^4$ computing its Clifford index $\gamma = 5$.
By \cite[3.2.2]{cm} the curve $X$ has gonality $d_1 = \gamma + 3 = 8$. Such curves exist and the 
$g_{13}^4$ is the only linear series on it computing $\gamma$ (see \cite[Example 3.2.7]{cm}). Hence
$d_2 \geq 10, \; d_3 = 12, \; d_4 = 13$ and the $g_{13}^4$ embeds $X$ into $\PP^4$. According to \cite[Lemma 4]{m1}
$X$ has trisecant lines in $\PP^4$. The projection with center such a line induces a $g_{10}^2$ on $X$. So $d_2 = 10$.
By Serre duality we get $d_5 = 16$ and $d_6 = 18$. Finally, for $r > 6 = g - d_1$ we know $d_r$ by 
Lemma \ref{lem2.2} (a) and (b). 
} $\hspace{2.1cm} \square$
\end{ex}

Note that using Serre duality it is easy to calculate the $d_r \geq g$ with $r \leq g - d_1$ provided all $d_s$ 
with $d_s < g$ are already known.\\

Let $M_g$ denote the moduli space of smooth projective curves of genus $g$ ($\geq 4$). Considering the invariants $d_r$ 
as functions $d_r: M_g \ra \ZZ$, we have

\begin{prop}
For any $r \geq 1$ the function $d_r: M_g \ra \ZZ$ is lower semi-continuous, i.e. may become smaller under specialisation.
\end{prop}

\begin{proof}
We have to show that for any $d$ the set
$$
M_{(g,r,d)} := \{ X \in M_g \;|\; d_r(X) \leq d \}
$$
is closed in $M_g$. 

Fix an integer $n \geq 3$ and let $M_g^n$ denote the moduli variety of curves of genus $g$ with level-$n$ structure. There
exists a universal curve 
$$
\pi_g^n: \cC_g^n \ra M_g^n
$$ 
of genus $g$ with level-$n$ structure (which we omit in the notation, since we don't need it).   
According to \cite{gr}, for any integer $d$ the relative Picard scheme $\Pic^d_{\cC_g^n/M_g^n}$ is projective over 
$M_g^n$. Since it represents a functor, there exists a universal line bundle $\cL_d$ on 
$\cC_g^n \times_{M_g^d} \Pic^d_{\cC_g^n/M_g^n}$ of degree $d$. For any geometric point $x \in \Pic^d_{\cC_g^n/M_g^n}$ we denote by
$L_x$ the corresponding line bundle on the curve $C_{p(x)}$ corresponding to the point $p(x)$, where 
$p: \Pic^d_{\cC_g^n/M_g^n} \ra M_g^n$ denotes 
the structure morphism. Now it is well known that the function
$$
h^0: \left\{ \begin{array}{ccc}
             \Pic^d_{\cC_g^n/M_g^n} & \ra & \ZZ\\
             x & \mapsto & h^0(L_x)
             \end{array} \right.
$$
is upper semi-continuous.

Let $C \mapsto C_0$ denote a specialization in $M_g^n$.
Since the map $M_g^n \ra M_g$ is finite, it suffices to show that 
$$
d_r(C_0) \leq d_r(C).
$$
Let $L$ be a line bundle on $C$ computing $d_r$.  
Since $p$ is a projective morphism, the specialization $C \mapsto C_0$ lifts to a specialization $L \mapsto L_0$. By upper
semicontinuity of $h^0$ we have $h^0(L_0) \geq h^0(L)$, which implies that $d_r(C_0) \leq d_r(C)$. 
\end{proof}

\section{The slope inequality}

As an immediate consequence of Lemma \ref{lem2.2} (a),(b),(c) we get the following proposition.

\begin{prop} \label{prop3.1}
{\em (i)} $\frac{d_{g-1}}{g-1} = 2 = \frac{d_g}{g}$;\\
{\em (ii)} $\frac{d_r}{r} > \frac{d_{r+1}}{r+1}$ for any $r > g - d_1,\; r \neq g-1$;\\
{\em (iii)} $\frac{d_r}{r} > \frac{d_{r+1}}{r+1}$
for a general curve $X$ and any $r, \;1 \leq r \neq g- 1$.
\end{prop}

We call the inequality 
\begin{equation} \label{eq3.1}
\frac{d_r}{r} \geq \frac{d_{r+1}}{r+1}
\end{equation} 
a {\it slope inequality} (for the gonality sequence) and we say that $X$ {\it satisfies the slope inequalities}
if \eqref{eq3.1} is valid for all $r \geq 1$.  By Proposition \ref{prop3.1} this is valid for general curves. 
It is not difficult to see from \cite[Remark 5.5]{ln} that also hyperelliptic, trigonal, general tetragonal and bielliptic curves
satisfy the slope inequalities.
The following proposition shows that this is also true for general 
pentagonal curves.

\begin{prop} \label{prop4.2}
For a general pentagonal curve $X$ the gonality sequence is given by
$$
d_r = \left\{ \begin{array}{ccc}
              5r & & r \leq \lfloor \frac{g-3}{5} \rfloor,\\
              \lceil \frac{5r +  g -3}{2} \rceil &\quad \mbox{for}\quad& \frac{g-3}{5} < r \leq \lfloor \frac{g-1}{5} \rfloor,\\
              r+g-1 - \lfloor \frac{g-r-1}{4} \rfloor &\quad \mbox{for} \quad & \frac{g-1}{5} < r \leq g-1,\\
              r+g && r \geq g.
              \end{array} \right.
$$
In particular, $X$ satisfies the slope inequalities.
\end{prop}

\begin{proof}
This follows from \cite{p} and, for $r > \frac{g-1}{5}$, Serre duality and Riemann-Roch. The proof of 
the last assertion is an immediate computation. 
\end{proof}

For smooth plane curves we have however,

\begin{prop} \label{prop3.2}
Let $X$ be a smooth plane curve of degree $d \geq 5$. Then the slope inequality holds for all $r$ except if 
$$
r = \frac{\alpha(\alpha +3)}{2}  \quad \mbox{with} \quad 1 \leq \alpha \leq d - 4,
$$
in which case
$$
d_r = \alpha d \quad \mbox{and} \quad d_{r+1} = (\alpha+1)d - (\alpha +1). 
$$
So the slope ineqality \eqref{eq3.1} is not valid for these values of $r$.
\end{prop}

\begin{proof}
Noether's Theorem \cite[Theorem 3.14]{ci} says that
$$
 d_r = \left\{ \begin{array}{ccl}
                \alpha d - \beta & \textit{if} & r < g = \frac{(d -1)(d - 2)}{2},\\
                r + \frac{(d -1)(d - 2)}{2}  & \textit{if} & r \geq g.
                \end{array} \right.
 $$      
where $\alpha$ and $\beta$ are the uniquely determined integers with 
 $\alpha \geq 1$ and $0 \leq \beta \leq \alpha$ such that
 $$
 r = \frac{\alpha(\alpha + 3)}{2} - \beta.
 $$
In particular we have $d_1 = d - 1, \; d_2 = d,\; d_{g-1} = 2g-2,\; d_g = 2g$. Moreover, apart from $r = 1,g-1$ and
the exceptional values of the proposition we have $d_{r+1} = d_r + 1$. So for all these values of $r$ the slope 
inequality holds.

Now let $r = \frac{\alpha(\alpha +3)}{2}  \;\mbox{with} \; 1 \leq \alpha \leq d - 4$. Then 
$$
r < \frac{(d-3)d}{2} = g -1
$$ 
which gives the values of $d_r$ and $d_{r+1}$. Now $\frac{d_r}{r} < \frac{d_{r+1}}{r+1}$ is equivalent to
$\alpha^2 + (4-d)\alpha + (3-\alpha) < 0$ which is true, since $\alpha \leq d -4$.
\end{proof}

\begin{rem} {\em
Note that the biggest $r$ violating the slope inequality in Proposition \ref{prop3.2} is
$$
r = \frac{(d-4)(d -1)}{2} = g - (d -1) = g - d_1.
$$ 
So in this case part (ii) of Proposition \ref{prop3.1}
is best possible.

This is, however, the only case:} If $X$ is not a smooth plane curve, we always have 
$$
d_{g-d_1} = 2g-2 -d_1 = (g-d_1) + g -2
$$
which implies that $\frac{d_r}{r} \geq \frac{d_{r+1}}{r+1}$ for $r = g -d_1$, by Lemma {\em \ref{lem2.2} (b).

In fact, the dual of a pencil $g_{d_1}^1$ is a series of degree $2g-2-d_1$ and dimension $g-d_1$. Assume that 
$d_{g-d_1} < 2g-2-d_1$. Then we can find a series of degree $2g-3-d_1$ and dimension $g-d_1$. But its dual is a $g_{d_1+1}^2$ 
which is very ample, since otherwise the subtraction of two appropriate points of $X$ would give us a $g_{d_1 -1}^1$ on $X$.  
Hence $X$ is a smooth plane curve of degree $d_1+1$.
} $\hspace{8cm} \square$
\end{rem}

Now let $X \subset \PP^r$ be an extremal curve of degree $3r-1$. We may assume that $X$ 
is not isomorphic to a smooth plane curve. According to 
\cite[Section III, Corollary 2.6 (iii)]{acgh}, $X \subset \PP^r$ is a semicanonical curve of genus genus $g = 3r = d+1$ 
(see also \cite{a1} for these curves).

\begin{prop} \label{prop3.3}
With these assumptions we have
$$
d_r = d= g-1 \; \mbox{and} \;\; d_{r+1} = g+2 \; \mbox{for any} \; r \geq 2. 
$$
\end{prop}

An immediate consequence is,

\begin{cor} \label{cor3.4}
For any $r \geq 2$ and any extremal curve $X$ of degree $d = 3r-1$ in $\PP^r$ we have
$$
\frac{d_r}{r} < \frac{d_{r+1}}{r+1},
$$
i.e. \eqref{eq3.1} is violated.
\end{cor}

\begin{proof}
For $r =2$ this is a special case of Proposition \ref{prop3.2}. The only other case (see \cite{acgh}) where $X$ may be isomorphic 
to a smooth plane curve is $r=5$: a smooth plane septic is, by $2g_7^2$, also an extremal curve of degree $14$ in $\PP^5$. 
In this case the assertion is again a special case of 
Proposition \ref{prop3.2}.

So we may assume $r \geq 3$ and $X$ is not isomorphic to a smooth plane curve.
According to \cite[Section III, Theorem 2.5]{acgh}) $X$ lies on a rational normal scroll
surface, whose ruling sweeps out a $g_4^1$ on $X$. In particular $X$ admits a 4-secant line $\ell$. 
Projection with center $\ell$ induces a $g_{g-5}^{r-2}$. Consequently, $X$ has also a $g_{g-4}^{r-2}$ and thus 
by dualizing a $g_{g+2}^{r+1}$.

Assume that $X$ admits a $g_{g+1}^{r+1}$. By dualizing, $X$ has a $g_{g-3}^{r-1}$. 
But then, by \cite[Lemma 5.1]{a2}, 
\begin{equation} \label{eq4.2}
\dim |g_{g-1}^r + g_{g-3}^{r-1}| \geq r + 2(r-1),
\end{equation}
which contradicts Clifford's Theorem (since $g = 3r$). Hence $X$ does not admit a $g_{g+1}^{r+1}$ which implies $d_{r+1} = g+2$. 

Since $X$ has no $g_{g+1}^{r+1}$, it does not admit a $g_{g}^{r+1}$. Dualizing we see that $X$ has no 
$g_{g-2}^r$. This implies $d_r = g-1$.
\end{proof}

\begin{ex} \label{ex3.5}
Let $X$ be an extremal curve of degree $8$ in $\PP^3$. Then $g=9$ and the gonality sequence of $X$ is
\begin{center}
\begin{tabular}{|c||c|c|c|c|c|c|c|c|c|}\hline
$r$&$1$&$2$&$3$&$4$&$5$&$6$&$7$&$8$&$r \geq 9$\\
\hline
$d_r$&$4$&$7$&$8$&$11$&$12$&$14$&$15$&$16$&$r+9$\\
\hline
\end{tabular}
\end{center}
In particular, $\frac{d_3}{3} = \frac{8}{3} < \frac{d_4}{4} = \frac{11}{4}$.
\end{ex}

\begin{proof}
Since $X$ is tetragonal, we have $d_1=4$. From \eqref{eq4.2} we conclude that $d_2 = 7$ and by 
Proposition \ref{prop3.3}, $d_3 = 8$ and $d_4 = 11$. The dual series of a $g_4^1$ is a $g_{12}^5$. This implies
$d_5 = 12$ using Lemma \ref{lem2.1} (a).
The other assertions follow from Lemma \ref{lem2.2} (a) and (b).
\end{proof}

Note that for a general tetragonal curve of genus 9 we have $d_2 = 8$ and $d_3 = 10$, whereas the other values of $d_r$ coincide 
with the those of Example \ref{ex3.5} (see \cite[Remark 4.5 (c)]{ln}). On the other hand, for a bielliptic curve of genus 9, 
the values of $d_r$ are as in Example \ref{ex3.5}, apart form $d_2 = 6$ and $d_4 = 10$.\\

It is not difficult to see that, apart from smooth plane quintics, all curves of genus $g \leq 8$ satisfy 
the slope inequalities.\\ 

There is a more general principle showing that the curves of Propositions \ref{prop3.2} and \ref{prop3.3} do not satisfy 
all slope inequalities, namely the simple

\begin{lem} \label{lem3.6}
Let $X$ be a curve admitting a $g_d^r$ with $d \geq 2r-1 \geq 3$ such that \\
{\em (1)} $d_{r-1} = d-1$ and\\
{\em (2)} $2d \leq g+3r -2$.

Then $d_r = d$, the $g_d^r$ is complete and very ample and, if $g_{d'}^{r'}$ denotes the Serre-dual linear 
system $|K_X - g_d^r|$ of the $g_d^r$, we have $r' \geq r$ and
$$
\frac{d_{r'}}{r'} < \frac{d_{r'+1}}{r'+1}.
$$
\end{lem}

(Note that $d' = 2g-2-d$ and $r' = g-1-d+r$.) 

\begin{proof}
Clearly $d_r \leq d$ and since $d_{r-1} = d-1$, we have $d_r = d$. If the $g_d^r$ were incomplete or not very ample, the curve 
$X$ would admit a $g_{d-2}^{r-1}$. Then $d_{r-1} \leq d-2$, a contradiction.

Since the $g_d^r$ is simple and $d \geq 2r-1$, it follows from Castelnuovo's count (\cite[Lemma 3.2]{a2} 
that $\dim (2g_d^r) \geq 3r-1$, and assumption (2) implies that 
$$
\dim (2g_d^r) = 2d -g+1 + \dim |K_X - 2g_d^r| \leq 3r-1 + \dim |K_X - 2g_d^r|.
$$
Hence $2g_d^r$ is a special linear series, and we obtain
\begin{eqnarray*}
r' = \dim |K_X - g_d^r| &=& \dim |(K_X - 2g_d^r) + g_d^r|\\
& \geq &\dim |K_X - 2g_d^r| + \dim(g_d^r) \geq r.
\end{eqnarray*}
Since $|K_X - g_d^r|$ is a $g_{d'}^{r'}$, we know that $d_{r'} \leq d'$.

Assume that $d_{r'+1} < d'+3$. Then $X$ admits a complete $g_{d'+2}^{r'+1}$ which by dualization gives a 
$g_{d-2}^{r-1}$ contradicting $d_{r-1} = d-1$. Hence we have $d_{r'+1} \geq d' + 3$.

Since, by (2),
$$
d' = d -2r + 2r' < g-1-d+r +2r' = 3r',
$$
we clearly have
$$
\frac{d'}{r'} < \frac{d'+3}{r'+1}.
$$
So the inequalities $d_{r'} \leq d'$ and $d'+3 \leq d_{r'+1}$ of above complete the proof of the proposition.
\end{proof}

\begin{rem} {\em (i) In Lemma \ref{lem3.6} we have $r' > r$ unless $|2g_d^r| = |K_X|$.
In fact, otherwise we have by \cite[Lemma 5.1]{a2}, 
$$
g-1 = \dim |g_d^r + (K_X -g_d^r)| \geq r + 2r' = 2g-2 -2d + 3r
$$
contradicting (2).

In particular, $d \leq g-1$ and hence $d' \geq g-1$. Since we know from the proof of Lemma \ref{lem3.6} that
$d_{r'+1} \geq d'+3$, we get $d_{r'+1} \geq (g-1)+3 = g+2$. 
Hence Lemma \ref{lem3.6} cannot discover a violation of the slope inequality for numbers $< g$ of the gonality sequence.

(ii) Lemma \ref{lem3.6} implies that $d \geq 3r-1$. In fact, we noted in Section 2 that $g \leq \pi(d,r) < d$ for $d < 3r-1$.
$\hspace{6.1cm} \square$
}
\end{rem}

\begin{prop} \label{prop3.8}
Let $g_d^r$ be a very ample linear series on $X$ with $d \geq 3r-1$ and assume that 
$$
g > \pi(d,r) - \chi(d,r) + 2.
$$ 
Then  we have
\begin{enumerate}
\item $d_{r-1} = \min\{(r-1)d_1,d-1\}$, and
\item $2d \leq g +3r -2$.
\end{enumerate}
\end{prop}

\begin{proof} 
For $r = 2$ the assertions are obvious. So suppose $r \geq 3$.

(1):
Since $d_r \leq d$, clearly $d_{r-1} \leq d-1$.
By \cite[Lemma 5.1]{a2} we have
$$
\dim ( g_d^r + g_{d_{r-1}}^{r-1} ) \geq r + 2(r-1) = 3r-2.
$$
Assume that we even have $\dim ( g_d^r + g_{d_{r-1}}^{r-1} ) \geq 3r-1$. Since by assumption $\chi(d,r) \geq 3$ and
$g > \pi(d,r) - \chi(d,r) + 2$, it follows from Theorem \ref{thm2.3} that $g_d^r + g_{d_{r-1}}^{r-1} = 2g_d^r$,
which is absurd. Hence we have 
$$
\dim(g_d^r + g_{d_{r-1}}^{r-1}) = 3r-2.
$$ 
Assume that $d_{r-1} \leq d-2$. We have to show that $d_{r-1} = (r-1)d_1$.

If $g_{d_{r-1}}^{r-1}$ is simple, then $ g_d^r - g_{d_{r-1}}^{r-1}  \neq \emptyset$ by
\cite[Lemma 5.2]{a2}. But since $d_{r-1} \leq d-2$, this implies that $g_d^r$ is not very ample, a contradiction.
Hence $g_{d_{r-1}}^{r-1}$ is compounded. By \cite[Lemma 5.3]{a2} we conclude that 
$$
g_{d_{r-1}}^{r-1} = (r-1)g_t^1
$$
for a pencil $g_t^1$ on $X$. Clearly $t \geq d_1$ and, since $(r-1)d_1 \geq d_{r-1} = (r-1)t$ by Lemma \ref{lem2.1} (c), 
we must have $t = d_1$.

(2): In order to abbreviate, we write $m := \chi(d,r)$. Then $d-1 = m(r-1) + p$ with $1 \leq p \leq r-1$. 
Since $d \geq 3r-1$ we have $m \geq 3$. Consequently, $(m-3)(r-1) + 2p \geq 2$ which implies that 
$$
\frac{m-2}{2} \left((m-3)(r-1) + 2p \right) \geq m-2.
$$ 
But one easily computes that
\begin{eqnarray*}
\pi(d,r) -(2d+1-3r) & =& (m-2) \left(d-1 - \frac{m+3}{2}(r-1)\right)\\
&=& \frac{m-2}{2} \left( (m-3)(r-1) + 2p \right)  \geq m-2.
\end{eqnarray*}
This implies that $g > \pi(d,r) -m + 2 \geq 2d + 1 -3r$,
which is (2). 
\end{proof}

\begin{ex} \label{ex3.9}
$(d=10, r=3)$. {\em A smooth curve $X$ of type $(5,5)$ on a smooth quadric $Q$ in $\PP^3$ has genus $g = 16 = \pi(10,3)$ and
gonality $d_1 = 5$ (see \cite{m}). A smooth curve of type $(4,6)$ on $Q$ has genus $g = 15$ and gonality $d_1 = 4$ (see \cite{m}).
By Proposition \ref{prop3.8}, $d_2 = 9 = \deg X -1$ in the first case and $d_2 = 2d_1 = 8$ in the second case. (Note, 
however, that in the latter case $X$ is an extremal curve of degree 14 in $\PP^5$.)} $  \hspace{0.5cm} \square$
\end{ex}

Example \ref{ex3.9} can be generalized as follows:

\begin{ex} \label{ex4.12}
{\em Let $X$ be a smooth curve of type $(a,b)$ with $a \leq b$ on a smooth quadric $Q$ in $\PP^3$. Then $X$ has 
degree $d = a+b$, genus $g=(a-1)(b-1)$
and gonality $d_1 = a$ and, according to \cite[Example 4.9]{cc}, we have $d_2 = 2a$ if $a < b$ and $d_2 = d-1$ if $a=b$.
In the latter case we have $g = (a-1)^2$ and so condition (2) in Lemma \ref{lem3.6} just reads $(a-3)^2 > 0$. 
Hence for $a \geq 4$ Lemma \ref{lem3.6} implies that a smooth curve of type $(a,a)$ on $Q$ does not satisfy all slope 
inequalities. Note that such a curve is a smooth complete intersection of $Q$ with a surface of degree $a \geq 4$ and it is an
extremal space curve of degree $2a$.} $  \hspace{1.5cm} \square$
\end{ex}

By Proposition \ref{prop2.4} the curve $X$ in Proposition \ref{prop3.8} is a smooth curve of degree $d$ on a surface 
of degree $r-1$ in $\PP^r$. In this case it is not difficult to compute $d_1$ (see \cite{m}).

\begin{theorem} \label{thm3.10}
Let $X$ be an extremal curve of degree $d \geq 3r-1$ in $\PP^r$. Then $d_{r-1} = d-1$ and $X$ does not satisfy all 
slope inequalities.
\end{theorem}

\begin{proof}
By Lemma \ref{lem3.6} and Proposition \ref{prop3.8} it is enough to show that $(r-1)d_1 \geq d-1$ for an extremal 
curve of degree $d$ in $\PP^r$. 

This is true for $r=2$. For $r \geq 3$ it is known (combine \cite[Section III, Theorem 2.5]{acgh} and \cite{m}) that 
$$
d_1 = \chi(d,r) + 1, \quad  \mbox{if} \; r-1\; \mbox{does not divide}\; d-1,
$$
and
$$
\chi(d,r) + 1 \leq d_1 \leq \chi(d,r) + 2, \quad \mbox{if} \; r-1\; \mbox{divides}\; d-1,
$$ 
unless $r=5$, in which case $X$ can also be isomorphic to a smooth 
plane curve of degree $\frac{d}{2}$, which has gonality $d_1 = \frac{d}{2} -1$. This implies the assertion.
\end{proof}

Recall that by Proposition \ref{genera} the genera of the curves of Theorem \ref{thm3.10} cover a big subset of ${\mathbb N}$.\\

The following example shows that a violation of the slope inequalities for members of the gonality sequence less than $g$ 
is not restricted to smooth plane curves.

\begin{ex} \label{ex3.11}    
A smooth curve $X$ of type $(5,5)$ on a smooth quadric $Q$ in $\PP^3$ is an extremal space curve of degree $10$
whose gonality sequence is for $r \leq g - d_1 -1 = 10$:
\begin{center}
\begin{tabular}{|c||c|c|c|c|c|c|c|c|c|c|}\hline
$r$&$1$&$2$&$3$&$4$&$5$&$6$&$7$&$8$&$9$&$10$\\
\hline
$d_r$&$5$&$9$&$10$&$14$&$15$&$18$&$19$&$20$&$23$&$24$\\
\hline
\end{tabular}
\end{center}
In particular $\frac{d_3}{3} < \frac{d_4}{4}$.
\end{ex} 

\begin{proof}
The first 3 members of the gonality sequence are clear by example \ref{ex3.9}. According to \cite[Lemma 5.1]{a2} 
the linear series $g_{10}^3 +g_5^1$ is a $g_{15}^n$ with $n \geq 5$, and $n > 5$ is impossible by \cite[Corollary 2.4.3]{cm}. 
By Serre duality it suffices to show that $d_4 = 14$. 

Assume that $X$ admits a $g_{13}^4$. Since $d_1 = 5$, we have $\dim (g_{13}^4 - g_{10}^3) \leq 0$.
Then \cite[III Exercise B-6]{acgh} implies that
$$
\dim (g_{13}^4 + g_{10}^3) \geq 2 \cdot 4 + 3 - 1 = 10.
$$
However a linear series $g_{23}^{10}$ is Serre-dual to a $g_7^2$ which contradicts $d_2 = 9$.
\end{proof}                                                  

By Section 2, the next example is not an extremal curve of degree $d \geq 3r-1$ in $\PP^r$.

\begin{ex} \label{ex3.12}
A smooth complete intersection $X$ of a cubic and a quartic surface in $\PP^3$ has degree $12$, genus $19$,
Clifford index $\gamma = 6 =d_1 -2$ and does not satisfy all slope inequalities. 
\end{ex}

\begin{proof}
By \cite{b} we have $d_1 =12 - \ell = \gamma +2$ if $X$ has an $\ell$-secant line, but no $(\ell +1)$-secant line.

According to \cite[Lemma 2]{m1}, $X$ has a $4$-secant line. It has no $5$-secant line, since such a line would lie on both 
the cubic and the quartic surface intersecting in $X$ and would thus be contained in $X$. Hence $d_1 = 8, \gamma = 6$ and $X$ has genus 19 
(\cite[III, Ecercise C-1]{acgh}).

Concerning the statement on the slope inequalities, according to Lemma \ref{lem3.6} it is enough to show that $d_2 = 11$.
So assume that $X$ has a $g_{10}^2$ (computing $\gamma$). Then, by \cite[Lemma 5.1]{a2} we have 
$\dim (g_{12}^3 + g_{10}^2) \geq 7$. If equality holds, we get the contradiction $d_2 = 2d_1 = 16$ in the same way as in the 
last part of the proof of Proposition \ref{prop3.8} (1). So $g_{12}^3 + g_{10}^2$ gives rise to a $g_{22}^8$
on $X$, whose dual is a $g_{14}^4$ computing $\gamma$. But according to \cite[Corollary 3.2.5]{cm} there is no $g_d^r$ on $X$ 
computing $\gamma$ such that $12 < d < g = 19$. Hence $d_2 = 10$ is impossible. 
\end{proof}

In the Examples \ref{ex3.3}, \ref{ex3.5} and \ref{ex3.12} one can compute the Clifford index $\gamma$ by a linear 
series which is not a pencil. The following proposition is a consequence of 
Lemma \ref{lem3.6} for a curve with such a property. 

\begin{prop}  \label{prop4.14}
Assume that $g \geq 2 \gamma + r + 2$ and $d_r = \gamma + 2r$ for some $r \geq 2$. If $X$ satisfies all 
slope inequalities, then for all $s = 1, \ldots, r$,
$$
d_s = \gamma + 2s.
$$  
\end{prop}
 
\begin{proof}
By our assumption on $g$ we have $g + s - 1 > \gamma +2s$ for all $s = 1, \ldots, r$. So $\min \{ \gamma +2s, g+s-1 \} = \gamma +2s$.
According to Lemma \ref{lem2.2} (d) we thus have $d_s \geq \gamma + 2s = d_r - 2(r-s)$ for $s = 1, \ldots, r$.

Hence, if $X$ satisfies all slope inequalities, then Lemma \ref{lem3.6} implies that 
$d_{r-1} = d_r -2$. Repeating this argument gives the assertion.
\end{proof} 

The curves fulfilling the hypotheses of Proposition \ref{prop4.14} are good candidates for curves not satisfying 
all slope inequalities. We will demonstrate this in the next section.

\section{Curves on a general K3-surface}

We extended Corollary \ref{cor3.4} to Theorem \ref{thm3.10}. Likewise, the curves of Example \ref{ex3.12} generalize 
to a bigger class of curves not satifying all slope inequalities. We consider the following situation.

Let $S$ be a general K3-surface of degree $2r-2$ in $\PP^r$. Then $\Pic S$ is generated by the class   
of a hyperplane section $H$ and $H^2 = \deg S = 2r-2$.
Let $X$ be a smooth irreducible curves in $S$. Then $X$ is contained in a linear series $|nH|$ for some  
positive integer $n$. The curve $X$ is $\frac{1}{n}$-canonical (i.e. $K_X = \cO_X(n)$), of genus 
$g = \frac{1}{2}X^2 + 1 = n^2(r-1) + 1$ and degree $d = X\cdot H = 2n(r-1)$ in $\PP^r$.

If $n \geq 2$, then by \cite[Example 3.2.6]{cm} the linear system $|H|_X|$ computes the Clifford index 
$\gamma = d-2r = 2(n-1)(r-1) -2$ of $X$. For $n = r = 3$ we obtain a special case of Example \ref{ex3.12}.

\begin{lem} \label{lem5.1}
In the situation just described, assume that $n \geq 4$. Let $D$ be a base point free and effective divisor on $X$ of 
degree $\delta < d$.
Then $D$ is contained in a hyperplane section of $X$ (more precisely, $h^0(H|_X - D) = 1$ if $D \neq 0$).
\end{lem}

\begin{proof}
We apply Reider's method \cite{re}: according to \cite[Proposition 2.10, Remark 2.11, 1) and Corollary 1.40]{re}, 
for any base point free and effective divisor $D$ on $X$ of degree $\delta < \frac{1}{4}X^2$ there is an 
effective non-trivial divisor $E_1$ on $S$ such that 
$$
E_1^2 < (X -E_1)\cdot E_1 \leq \delta, \quad \mbox{and} \quad  E_1|_X - D \geq 0.
$$

We apply this to our divisor $D$ of the statement of the lemma. We have
$$
\delta < d = 2n(r-1) \leq \frac{1}{2} n^2(r-1) = \frac{1}{4}X^2
$$
for $n \geq 4$. So by Reider's method, there is a divisor $E_1$ on $S$ with the indicated properties.
Since $\Pic S = \ZZ H$, we have $E_1 \in |\lambda H|$ for some positive integer $\lambda$. Then 
$E_1^2 < (X - E_1)\cdot E_1$ implies that $\lambda < \frac{n}{2}$, and from 
$(X - E_1) \cdot E_1 \leq \delta < d$ we obtain
$$
\lambda^2 -n \lambda + n \geq 1.
$$
Since $\lambda$ in an integer, this is equivalent to
$$
\left(\lambda - \frac{n}{2} \right)^2 \geq \left(\frac{n-2}{2} \right)^2.
$$ 
But $\lambda - \frac{n}{2} < 0$ and $\frac{n-2}{2} > 0$, so this implies that $\lambda - \frac{n}{2} \leq - \frac{n-2}{2}$, i.e.
$\lambda \leq 1$. Consequently, $|E_1| = |H|$, and so $E_1|X -D \geq 0$ shows that $h^0(H_X -D) \geq 1$.
To prove equality here, observe that $\delta \geq d_1 \geq \gamma + 2$ if $D \neq 0$. Then
\begin{eqnarray*}
\deg (H|_X -D) = d - \delta & \leq & d - (\gamma+2) = (\gamma +2r) -(\gamma+2)\\
& = & 2(r-1) = \frac{\gamma+2}{n-1} < \gamma + 2 \leq d_1.
\end{eqnarray*}
This implies that $h^0(H|_X - D) =1$.
\end{proof}

\begin{cor} \label{cor5.2}
In the situation of above let $n \geq 4$. Then $|H|_X|$ is the only $g_d^r$ on $X$.
\end{cor}

\begin{proof}
Let $|L|$ be a $g_d^r$ on $X$ different from $|H|_X|$. Since $|L|$ computes $\gamma$,
it is base point free and simple (\cite{cm}). Hence for a general divisor $E$ in $|L|$ there is a point $p \in X$ 
in the support of $E$ such that the divisor $E - p$ is base point free. So Lemma \ref{lem5.1} implies that $E -p$ 
is contained in $|H|_X|$. But by the General Position Theorem (\cite[Theorem 4.1]{a2}), the greatest common divisor
of $E$ and any divisor in $|H|_X|$ has degree at most $r$. Since $r < d-1$, we get a contradiction.
\end{proof}

By the way, by Theorem \ref{thm2.3} also all curves of Theorem \ref{thm3.10} have merely one $g_d^r$.

Combining Lemma \ref{lem3.6} (or Proposition \ref{prop4.14}) with Lemma \ref{lem5.1} we obtain our second main result.

\begin{theorem} \label{thm5.2}
Let $X$ be a curve contained in a general K3-surface of degree $2r-2$ in $\PP^r$. If 
$d := \deg X \geq 8(r-1)$, then $d_{r-1} = d-1$ and $X$
does not satisfy all slope inequalities.
\end{theorem}

\begin{proof}
A $g_{d_{r-1}}^{r-1}$ on $X$ is contained in the $g_d^r = |H|_X|$, by Lemma \ref{lem5.1}.
If $d_{r-1} \leq d-2$, this contradicts the very ampleness of the $g_d^r$. Hence $d_{r-1} = d-1$, and we can apply Lemma \ref{lem3.6}.
\end{proof}

According to Dirichlet's prime number theorem, the function
$$
f(r) := n^2(r-1) + 1
$$
represents, for every integer $n \neq 0$, infinitely many prime numbers. Hence, for fixed $n \geq 4$, Theorem \ref{thm5.2} 
produces an infinite number of curves of prime genus. By Section 2, this cannot be achieved by virtue of Theorem \ref{thm3.10}.
However, this does not yet answer the following question:

\vspace{0.2cm}
\noindent
{\bf Question 5.4}: Is there a number $g_0 \geq 9$ such that for every integer $g \geq g_0$ there is a curve of 
genus $g$ not satisfying all slope inequalities?

\section{Space curves}

The curves of Examples \ref{ex3.5}, \ref{ex4.12} (with $a=b$), \ref{ex3.11} and \ref{ex3.12} are complete intersections 
of two surfaces in $\PP^3$. More generally, based on \cite{cc}, we have the following theorem.

\begin{theorem} \label{thm6.1}
Let $X$ be a smooth complete intersection of two surfaces of degree $p$ and $s$ with $p \geq s \geq 2$ in $\PP^3$. Then 
$X$ is a space curve of degree $d:= ps$ and genus $g= \frac{1}{2}ps(p+s-4) + 1$, and we have $d_2 = d-1$. 
If $p \geq 4$, then $X$ does not satisfy all slope inequalities. 
\end{theorem}

\begin{proof}
Since $\cO_X(p+s-4)$ is the canonical bundle of $X$, we have $2g-2 = d(p+s-4)$ (\cite[III, Exercise C-1]{acgh}).
In \cite[Example 4.6]{cc} it is shown that $d_2 = d-1$. The condition (2) in Lemma \ref{lem3.6} is equivalent 
to 
$$
ps(8-p-s) < 18,
$$ 
and this is fulfilled if $p \geq 4$ (i.e if $(s,p) \neq (2,2), (2,3), (3,3)$).
So Lemma \ref{lem3.6} gives the last assertion of the theorem.
\end{proof}

By \cite{b} the curve $X$ in Theorem \ref{thm6.1} has gonality $d_1 = d-\ell$ if $X$ has an $\ell$-secant line, 
but no $(\ell + 1)$-secant line. By \cite{n}, $\ell \leq s$ or $\ell = p$.
Let $p \geq 4$. Then $\gamma = d_1 - 2$ (\cite{b}) and $\ell \geq 4$ ( \cite[Lemma 2]{m1}).
In fact, according to \cite[Corollary 5.2]{hs}, if $X$ is general 
in its linear series on a smooth surface $S$ of degree $s$ and if $S$ does not contain a line, then 
$s \geq 4$ and $\ell = 4$. So the 
linear series of plane sections computes $\gamma = d - 6$ for a ``general'' complete intersection $X$ 
of two surfaces of degrees $\geq 4$ in $\PP^3$.\\ 

The curve $X$ in Theorem \ref{thm6.1} is a special case of a Halphen curve which is defined as follows.
\begin{defin}
{\em A smooth and irreducible space curves $X$ of degree $d$ is called a {\it Halphen curve} if there is 
an integer $s \geq 2$ such that $d > s(s-1)$ and if $X$ is of maximal genus among those smooth 
and irreducible space curves of degree $d$ which do not lie on a surface of degree $< s$.

If $d = ks-\epsilon$ with $0 \leq \epsilon < s$, then $X$  has genus
$$
g = G(d,s) := \frac{d^2}{2s} + \frac{d(s-4)}{2} + 1 - \frac{\epsilon}{2}(s-1-\epsilon + \frac{\epsilon}{s}),
$$
which is Halphen's bound (\cite[Remark 4.3]{cc}, note the typo there). $\hspace*{0.4cm} \square$ 
}
\end{defin}

\begin{theorem} \label{thm6.3}
Let $X$ be a Halphen curve of degree $d$. Then $d_2 = d-1$ and if $d \geq 11$, then $X$ does not 
satisfy all slope inequalities.
\end{theorem}

\begin{proof}
According to Theorem \ref{thm6.1} and \cite[Theorem 4.7]{cc} we have $d_2 = d-1$. In order to apply Lemma \ref{lem3.6} 
we have to check that $g-1 > 2d-9$.

The function 
$$
f(\epsilon) := \frac{\epsilon}{2}(s-1-\epsilon + \frac{\epsilon}{s}) = \frac{s-1}{2s}\epsilon (s-\epsilon)
$$
is maximal for $\epsilon = \frac{s}{2}$. So $f(\epsilon) \leq \frac{1}{8}s(s-1)$ for $0 \leq \epsilon < s$.
Hence $g-1 > 2d-9$ is satisfied if 
$$
\frac{d^2}{2s} + \frac{d(s-4)}{2} -f(\epsilon) \geq \frac{d^2}{2s} + \frac{d(s-4)}{2} - \frac{s(s-1)}{8} > 2d-9,
$$
i.e. if 
$$
d^2 + s(s-8)d + 18s > \frac{1}{4}s^2(s-1).
$$
Recall that $d > s(s-1)$. So this inequality is easily seen to be valid for $s \geq 4$. For $s=2$ it is valid for 
$d \geq 8$ and for $s=3$ for $d \geq 11$.
\end{proof}

Note that the curves of Theorems \ref{thm3.10}, \ref{thm5.2}, \ref{thm6.1} and \ref{thm6.3} are projectively normal 
(= arithmetically Cohen-Macaulay) under their embedding considered. One may ask the following question:

\begin{qu}
{\em  
Does a smooth, irreducible
and projectively normal curve of degree $d$ in $\PP^r$ have $d_{r-1} = d-1$?}
\end{qu}


\begin{thebibliography}{CAV}
\bibitem{a1}
R.D.M. Accola:
\emph{Plane models for Riemann surfaces admitting certain half-canonical linear series}.
Proc. of the 1978 Stony Brook conference. Princeton Univers. Press (1981), 7-20.
\bibitem{a2}
R.D.M. Accola:
\emph{On Castelnuovo's inequality for algebraic curves I}.
Trans. AMS 251 (1979), 357-373.
\bibitem{acgh}
E. Arbarello, M. Cornalba, P. A. Griffiths and J. Harris: 
\emph{Geometry of Algebraic Curves I}. 
Springer, Grundlehren math Wiss. 267 (1985).
\bibitem{b}
B. Basili:
\emph{Indice de Clifford des intersections compl\`etes de l'espace}.
Bull. Soc. Math. France 124 (1996), 61-95.
\bibitem{cc} L. Chiantini, C. Ciliberto:
\emph{Towards a Halphen Theory of linear series on curves}.
Trans. AMS 351 (1999), 2197-2212.
\bibitem{ci} C. Ciliberto:
\emph{Alcune applicazione di un classico procedimento di Castelnuovo}.
Sem. di Geom., Dipart. di Matem., Univ. di Bologna, (1982-83), 17--43.
\bibitem{cm} M. Coppens and G. Martens: 
\emph{Secant spaces and Clifford's theorem}.
Comp. Math. 78 (1991), 193-212.
\bibitem{gr} A. Grothendieck:
\emph{Fondements de la G\'eom\'etrie Alg\'ebrique}.
S\'em Bourbaki 1957-1962, Paris (1962).
\bibitem{hs} R. Hartshorne, E. Schlesinger:
\emph{Gonality of a general ACM curve in $\PP^3$}.
arXiv: 0812.1634v1 (2008).
\bibitem{ln} H. Lange and P. E. Newstead: 
\emph{Clifford indices for vector bundles on curves}.
To appear in: A. Schmitt (Ed.) Affine Flag Manifolds and Principal Bundles. Birkh\"auser. 
\bibitem{m} G. Martens:
\emph{The gonality of curves on a Hirzebruch surface}.
Arch. Math. 67 (1996), 349-352.
\bibitem{m1} G. Martens:
\emph{\"Uber den Clifford-Index algebraischer Kurven}.
J. reine angew. Math. 336 (1982), 83-90.
\bibitem{n} S. Nollet:
\emph{Bounds on multisecant lines}.
Collect. Math. 49 (1998), 447-463.
\bibitem{p} S.-S. Park:
\emph{On the variety of special linear series on a general $5$-gonal curve}.
Abh. Math. Sem. Univ. Hamburg 72 (2002), 283-291.
\bibitem{re} I. Reider:
\emph{Some applications of Bogomolov's theorem}.
In: F. Catanese (Ed,): Problems in the theory of surfaces and their classification. Symposia Mathematica 32. 
Acad. Press (1991), 376-410.
\end{thebibliography}
\end{document}